\documentclass[12pt]{article}

\usepackage{amssymb,amsthm,amsmath}
\usepackage[usenames,dvipsnames]{xcolor}
\usepackage{soul}
\usepackage[normalem]{ulem}
\usepackage{cancel}
\usepackage{mathabx}
\allowdisplaybreaks
\usepackage{color,graphicx,epsfig}
\usepackage{mathrsfs}
\usepackage{enumerate} 
\usepackage{enumitem} 
\usepackage{esvect}

\newcommand{\comment}[1]{}

\definecolor{teal}{RGB}{0,128,128}
\definecolor{darkpurple}{RGB}{128,0,128}

\usepackage[titletoc,toc,title]{appendix}

\newtheorem{theorem}{Theorem}[section]

\newtheorem{cor}[theorem]{Corollary}

\theoremstyle{definition}

\theoremstyle{definition}

\newtheorem{example}[theorem]{Example}
\newtheorem{remark}[theorem]{Remark}

\def \cB {{\cal B}}

\def \cD {{\cal D}}

\title{Characterizing pyramidal Hadamard designs with the largest number of fixed points}

\author{
T.\ Traetta\footnotemark[1]
}
\date{\vspace{-5ex}}

\begin{document}
\maketitle

\footnotetext[2]{DICATAM, Universit\`{a} degli Studi di Brescia, Via Branze 43, 25123 Brescia, Italy. E-mail: tommaso.traetta@unibs.it}

\begin{abstract} 
A symmetric $(v,k,\lambda)$-design is said to be $f$-pyramidal, with $f<v-1$, under the action of a group $G$ if  $G$ acts as an automorphism group fixing $f$ points and acting sharply transitively on the remaining ones.
We show that, necessarily,
$f\leq v -2(k-\lambda)$. In particular, for a Hadamard design with parameters
$(v,k,\lambda)=(4m-1, 2m-1, m-1)$, for some $m\in\mathbb{N}$, it follows that $f\leq 2m-1$.

Recently, working on the complement design, the family $\cal P$ of Hadamard $(4m-1,2m-1,m-1)$-designs admitting an abelian $(2m-1)$-pyramidal automorphism group $G$ has been completely determined in the case $m=2^{k}>1$.

In this paper, we generalize that result by showing that $\mathcal{P}$ coincides with the family of Hadamard $(4m-1,2m-1,m-1)$-designs admitting a $(2m-1)$-pyramidal automorphism group, without assuming either that the group is abelian or that $m$ is a power of $2$.
\end{abstract}



\section{Introduction}

A \emph{$(v,k,\lambda)$-design} is a pair ${\cal D}=(P,\mathcal{B})$ where $P$ is a set of $v$ \emph{points} and $\mathcal{B}$ is a collection of $k$-subsets of $P$, called \emph{blocks}, such that every pair of distinct points of $P$ occurs in exactly $\lambda$ blocks. 

When the number $v$ of points equals the number $|\mathcal{B}|$ of blocks, $\cal D$ 
is called \emph{symmetric} of \emph{order} $n=k-\lambda$. 
In this case, letting ${\cal B}^c= \{P\setminus B\mid B\in {\cal B}\}$, 
it is known that ${\cal D}^c= (P, {\cal B}^c)$ 
is a symmetric $(v,v-k,v-2k+\lambda)$-design, called the \emph{complement} of $\cal D$; furthermore, 
$\cal D$ is said to be a \emph{Hadamard} design if 
$(v,k,\lambda)=(4m-1, 2m-1, m-1)$ and, in this case, its complement ${\cal D}^c$ is a symmetric $(4m-1, 2m, m)$-design.


It is well-known (see, for instance, \cite[Proposition 3.11]{BJL}) that for a symmetric
$(v,k,\lambda)$-design $\cal D$ to exist, we must have
$4n-1\leq v\leq n^2+n+1$.
Furthermore, the lower bound is achieved, that is, $v=4n-1$, if and only if
$(v,k,\lambda) = (4m-1, 2m-1, m-1)$ and $\cal D$ is a Hadamard design, or 
$(v,k,\lambda) = (4m-1, 2m, m)$ and $\cal D$ is the complement of a Hadamard design.

Letting $G$ be a group of automorphisms of a $(v,k,\lambda)$-design
$\cal D$, we say that $\cal D$ is \emph{$f$-pyramidal} under the action of $G$ if $G$ fixes exactly $f$ points and acts sharply transitively on the remaining $v-f$ points. The cases $f=0$ and $f=1$ correspond to sharply point-transitive and $1$-rotational designs, respectively, and their existence is known to be equivalent to the existence of suitable difference families (see, for instance, \cite{AbelBuratti}). A systematic study of $f$-pyramidal $(v,k,\lambda)$-designs has been carried out when $(k,\lambda)=(3,1)$, that is, 
for Steiner triple systems (see, for instance, \cite{BoBuRiTra12, BoBuGaRiTra21, Bu01, BuRiTra17, ChTrZh25, Mi07, PheRo81}). 
Much less is known about symmetric designs admitting an $f$-pyramidal automorphism group
when $f>0$. In the case $f=0$, a sharply point-transitive symmetric design is equivalent to the existence of a suitable difference set (see, for instance, \cite{JungPottSmith}).
A recent investigation in the case $f>0$ has been carried out in \cite{PankovAbelian} (see also \cite{PaPeZy}). In \cite{PankovAbelian}, the author determines the family of symmetric Hadamard $(4m-1, 2m-1, m-1)$-designs (and their complements) admitting a $(2m-1)$-pyramidal automorphism group $G$, under the assumptions that $G$ is abelian and $m$ is a power of $2$.

In this paper, we generalize this result by dropping the assumptions that $G$ is abelian and that $m$ is a power of $2$, thereby proving the main result of the paper, Theorem \ref{main} (Section \ref{proof}). In Section 2, we provide an upper bound on $f$ for the existence of an $f$-pyramidal symmetric design (Theorem \ref{bound}).

\section{An upper bound for pyramidal symmetric designs}

We recall here a classic result in Design Theory (see, for instance, \cite[Result 2.3.12]{Dem}).

\begin{theorem} \label{fixedpointsblocks}
Every automorphism of $\cal D$ has equally many fixed points and fixed blocks.
\end{theorem}

\begin{remark}\label{rem}
An automorphism of a symmetric design $\cD$ is also an automorphism of its complement; hence, $\cD$ is $f$-pyramidal under $G$ if and only if $\cD^c$ is. 
\end{remark}

For a symmetric design $\cal D$, the following result provides an upper bound on the number $f$ of points fixed by an automorphism group having an $f$-pyramidal action on the point-set of $\cal D$.

\begin{theorem}\label{bound}
If there is an $f$-pyramidal symmetric $(v,k,\lambda)$-design,
then $f\leq v -2(k-\lambda)$ or $f\in\{v-1,v\}$.
\end{theorem} 
\begin{proof} Let $P$ be the set of points of a symmetric $(v,k,\lambda)$-design having an $f$-pyramidal automorphism group $G$. Also, let $F\subset P$ be the set of $f$ points fixed by $G$. Since the assertion is straightforward when $f=v-1,v$, 
we can assume $f\leq v-2$; recalling that $G$ acts sharply transitively on $P\setminus F$, we have that $|G|=|P\setminus F| = v-f\geq 2$.
 
 Now, recall that every automorphism of $\cal D$ has equally many fixed points and fixed blocks (Theorem \ref{fixedpointsblocks}). Since an automorphism $\alpha\in G\setminus\{id\}$ fixes $f< v$ blocks, it moves at least one block $B$, that is, $\alpha(B)\neq B$; hence $|B\,\cap\, \alpha(B)|=\lambda$.
Since $B\,\cap\, F \subset \alpha(B)$, that is
 any point of $F$ lying in $B$
 must also lie in $\alpha(B)$, it follows that
 $P\setminus F \supset B\, \triangle\, \alpha(B)$, hence 
 \[v-f = |P\setminus F| \geq |B\setminus \alpha(B)|+|\alpha(B)\setminus B| = 
 2(k-\lambda),\]
 that is,  $f\leq v -2(k-\lambda)$.
\end{proof}

\section{Pyramidal Hadamard designs}\label{proof}
Restricting our attention to $f$-pyramidal Hadamard designs, the above upper bound on $f$ takes the following form.
\begin{cor} If there is an $f$-pyramidal Hadamard $(4m-1, 2m-1, m-1)$-design, then $f\leq 2m-1$.
\end{cor}

When $m=2^k>1$ and $f$ attains the upper bound, that is, 
$f=2m-1$, the family of Hadamard $(4m-1, 2m-1, m-1)$-designs admitting an $f$-pyramidal abelian automorphism group has been completely determined in \cite{PankovAbelian}. 
In Theorem \ref{main}, we provide a simplified and more general proof of this result, which allows us to characterize the Hadamard $(4m - 1, 2m - 1, m - 1)$-designs 
(hence, their complements) admitting a $(2m - 1)$-pyramidal automorphism group, thereby removing the assumptions that $m$ is a power of $2$ and that the group is abelian.

Before proving Theorem \ref{main}, we describe through the following example a method to construct a symmetric 
$(2^k-1, 2^{k-1}, 2^{k-2})$-design $\cal D = (P, \cB)$ (i.e., the complement of a Hadamard design) with center blocks (see \cite[Section 3]{PaPeZy2} and \cite[Example 8]{PankovAbelian}). Recall that a block $O\in \cB$ is said to be a \emph{center block} if 
$O\,\triangle\,B\in \cB$, for every block $B\in \cB$ distinct from $O$.

\begin{example} \label{ex:1}
Let $P$ be a set of size $2^k-1$ and let
$O\subset P$ with $|O| = 2^{k-1}$.
Choose a $(2^{k-1}-1)$-set $Z\subset O$, two symmetric 
$(2^{k-1}-1,  2^{k-2}, 2^{k-3})$-designs, 
say ${\cal D}_{O^c}=(O^c, \mathcal{B}_O)$ (where $O^c = P\setminus O$) and ${\cal D}_Z=(Z, \mathcal{B}_Z)$,
and a bijection $\delta: \mathcal{B}_O \to \mathcal{B}_Z$. Denoting by 
$\cB$ the set containing $O$ and all sets of the form
\[
X \cup \delta(X) \quad\text{and}\quad X \cup \bigl(O \setminus \delta(X)\bigr), \quad 
\]
for $X \in \mathcal{B}_O$, it is shown in \cite[Proposition 2]{PaPeZy2}  that $\cD = (P, \mathcal{B})$ is a symmetric $(2^k-1, 2^{k-1}, 2^{k-2})$-design where $O$ is a center block.
\end{example}

We refer to the design built in the previous example as the \emph{sum} 
of $D_{O^c}$ and $D_Z$.
The following result, proven in  \cite[Proposition 3]{PaPeZy2}, shows that any symmetric $(2^k-1, 2^{k-1}, 2^{k-2})$-design having a center block can always be constructed as in Example \ref{ex:1}.
\begin{theorem}\cite[Proposition 3]{PaPeZy2}\label{sum}
  Every symmetric $(2^k-1, 2^{k-1}, 2^{k-2})$-design with a center block is the sum of two symmetric $(2^{k-1}-1,  2^{k-2}, 2^{k-3})$-designs.
\end{theorem}

We also need the following characterization, proven in \cite{PankovAbelian}, of the symmetric designs consisting of the points and hyperplane complements of $PG(k-1,2)$.

\begin{theorem}\cite[Proposition 1]{PankovAbelian} \label{PG(k-1,2)}
   A symmetric $(2^k-1, 2^{k-1}, 2^{k-2})$-design is isomorphic to the design of points and hyperplane complements of $PG(k-1,2)$ if and only if for any distinct points $p,q$ there is a point $t$ such that there is no block containing $p,q,t$.
\end{theorem}

We are now ready to prove the main result of this paper.

\begin{theorem}\label{main}
Let $\cal D=(P, \cal B)$ be a Hadamard $(4m-1, 2m-1, m-1)$-design, with $m>1$. 
If $\cal D$ admits a $(2m-1)$-pyramidal automorphism group $G$, then
\begin{enumerate}
\item $G$ is an elementary abelian $2$-group of order $2m=2^{k}$ ($k>1$), and
\item ${\cal D}^c$ is the sum of
a symmetric $(2^{k}-1,2^{k-1},2^{k-2})$-design and the design of points and hyperplane complements of $\mathrm{PG}(k-1,2)$.
\end{enumerate}
\end{theorem}
\begin{proof} In view of Remark \ref{rem}, $G$ is a $(2m-1)$-pyramidal automorphism group of the complement design 
${\cal D}^c = (P, {\cal B}^c)$, a symmetric $(4m-1, 2m, m)$-design.

Let $F$ denote the set of points of ${\cal D}^c$ fixed by $G$ and note that
$|F|=2m-1$. Since $G$ acts sharply transitively on $P\setminus F$, it is known we can set $P\setminus F=G$ and $x+g=x$ for every $(x,g)\in F\times G$, in such a way that $B+g\in\cal B$, for every $B\in {\cal B}^c$. Furthermore, letting $\tau_g(p)=p+g$ for every $g\in G$ and $p\in P$, we have $\tau_g(B)=B+g$ and $\hat{G}=\{\tau_g\mid g\in G\}$ is a $(2m-1)$-pyramidal automorphism group of ${\cal D}^c$ isomorphic to $G$.

Let $B\in {\cal B}^c$ and assume there is $g\in G$ such that
$\tau_g(B)\neq B$. We start by showing that 
\begin{equation}\label{eq:1}
  \tau_g(B) = G\setminus B\,\cup\, (B\,\cap\, F).
\end{equation}
Notice that $|B\,\cap\, F|\leq m$. Indeed, recalling that $\tau_g$ fixes $F$ pointwise, we have that
\begin{equation}\label{eq:2}
\tau_g(B) = \tau_g(B\,\cap\, G)\,\cup\, \tau_g(B\,\cap\, F) 
           = \tau_g(B\,\cap\, G)\,\cup\, (B\,\cap\, F),
\end{equation}
hence $B\,\cap\,F\subset B\,\cap\,\tau_g(B)$. Recalling that two distinct blocks, such as $B$ and $\tau_g(B)$, share exactly $m$ points, it follows that $|B\,\cap\,F|\leq m$. Furthermore, since $\tau_g(B\,\cap\,G) \subset G$, we obtain from \eqref{eq:2}
that 
$\tau_g(B) \,\cap\, F = B \,\cap\, F \subset  B\,\cap\,\tau_g(B)$. It follows that
$G \supset B \,\triangle\, \tau_g(B)$, 
and taking cardinalities we get
\[2m = |G| \ge |B\, \triangle\, \tau_g(B)| = |B \setminus \tau_g(B)| + |\tau_g(B) \setminus B| = 2m,\]
since $|B|=|\tau_g(B)|=2m$ and $|B\cap \tau_g(B)|=m$. 
Therefore,
$G = B\, \triangle\, \tau_g(B)$ which means that 
$B\,\cap\,G$ and $\tau_g(B)\,\cap\,G = \tau_g(B\,\cap\,G)$ partition $G$. 
In other words, $\tau_g(B\,\cap\,G) = G\setminus(B\,\cap\,G) = G\setminus B$ which,
 when substituted into \eqref{eq:2}, gives $\eqref{eq:1}$.

 Given a block $B\in {\cal B}^c$, denote by $Orb(B)=\{B+g\mid g\in G\}$
 and $G_B=\{g\in G\mid B+g=B\}$  the orbit and the stabilizer of $B$, respectively, under the action of $\hat{G}$.
 Notice that if $B$ is fixed by $\hat{G}$ (i.e. $B+g=B$, for every $g\in G$)
 we must have $B=G$ (since $B\not\subset F$). Hence, in view of \eqref{eq:1},
\[
\text{Orb}(B) = 
\begin{cases}
  \{G\} & \text{if $B$ is fixed by $\hat{G}$},\\
  \{B, (G \setminus (B\, \cap\, G))\, \cup\, (B\, \cap\, F)\}
  & \text{otherwise}.
\end{cases}
\]
In particular, when $B\neq G$, we have $|G: G_B| = 2$, that is, $|G_B| =m$. 
Also,
since $(B\cap G)+G_B=B\cap G$, the set $B\cap G$ is a union of cosets of $G_B$ in $G$. However, $|B\cap G|<2m$ (since $B\neq G$) and $|G_B|=m$, hence
$B\cap G$ is a single coset of $G_B$, that is,
$B\cap G\in\{G_B,\; G\setminus G_B\}$. Therefore,
\[
  \text{Orb}(B) = \{G_B\, \cup\, (B \cap F),\; (G \setminus G_B) \cup (B\, \cap\, F)\}.
\]
We then must have $G\in \cal B$, otherwise $\cal B$ could be partitioned into orbits, each of cardinality $2$, thus contradicting the fact that $|{\cal B}^c|=4m-1$ is odd. Therefore,
${\cal B}^c=\bigcup_{i=1}^{2m-1} Orb(B_i) \,\cup\,\{G\}$, and $Orb(B_i)\neq Orb(B_j)$ whenever $i\neq j$. 
It follows that
the $2m-1$ stabilizers $G_{B_i}$ are pairwise distinct subgroups of $G$ of index $2$.
Indeed, if $G_{B_i}=G_{B_j}$, then $B_i\cap G$ and $B_j\cap G$ are cosets of the same subgroup of index $2$, hence $Orb(B_i)=Orb(B_j)$, that is, $i=j$.
Since each $G_{B_i}$ contains $2G$ (i.e., the subgroup of $G$ generated by $\{2g \mid g \in G\}$), they determine $2m-1$ pairwise distinct subgroups of index $2$ in $G/ 2G$.
Furthermore, since $G /  2G$ is an elementary abelian $2$-group, it is known it contains $|G / 2G |-1$ distinct subgroups of index $2$. Therefore, $2m-1\leq |G / 2G |-1$, thus implying that
$|G /  2G |=2m$ and hence $2G = \{0\}$. It follows that
$G$ is an elementary abelian group of order $2m=2^{k}$.

Finally, since $G\in \cal B$ is a center block of $\mathcal{D}^c$, by
Theorem \ref{sum} we have that the design $\mathcal{D}^c$ decomposes as the sum of 
two symmetric $(2^{k}-1,2^{k-1},2^{k-2})$-designs: more precisely, by
Example \ref{ex:1}, we have that 
${\cal D}^c$ is the sum of  
${\cal D}_{O^c}=(O^c, \mathcal{B}_O)$ and ${\cal D}_Z=(Z, \mathcal{B}_Z)$, with $O=G$ and $Z = G\setminus\{a\}$ (for some $a\in G$). 
It is left to show that
${\cal D}_Z$ is the design of points and hyperplane complements of $\mathrm{PG}(k-1, 2)$. Let $p, q\in Z$, with $p\neq q$ and set $g=a-p$.
If $B\in \mathcal{B}_Z$ contains $\{p,q,t=q+g\}$, then 
$\tau_g(B) \supset \{p+g, q+g, t+g\} = \{a,t, q\}$, thus contradicting the fact that $\tau(B) = G\setminus B$. The assertion follows from 
Theorem \ref{PG(k-1,2)}, and this completes the proof.
\end{proof}

\end{document}